\documentclass[12pt]{article}
\usepackage{amsfonts}
\usepackage{CJK}
\usepackage{amsthm}
\usepackage{amsmath}
\usepackage{hyperref}
\usepackage{indentfirst}
\usepackage[top=1em,bottom=1em,left=6em,right=6em,includehead,includefoot]{geometry}
\begin{document}
\begin{CJK}{GBK}{song}

\newtheorem{theorem}{Theorem}
\newtheorem{lemma}{Lemma}
\newtheorem{definition}{Definition}
\newtheorem{remark}{Remark}
\newtheorem{corollery}{Corollery}
\newtheorem{example}{Example}

\title{\bf A Remark on the Localization formulas about two Killing vector fields}
\author{Xu Chen \footnote{{\it Email:} xiaorenwu08@163.com. ChongQing, China }}
\date{}
\maketitle

\begin{abstract}
In this article, we will discuss a localization formulas of equlvariant cohomology about two Killing vector fields on the set of zero points
${\rm{Zero}}(X_{M}-\sqrt{-1}Y_{M})=\{x\in M \mid |Y_{M}(x)|=|X_{M}(x)|=0 \}.$ As application, we use it to get formulas about characteristic numbers and to get a Duistermaat-Heckman type formula on symplectic manifold.

\end{abstract}

The localization theorem for equivariant differential forms was
obtained by Berline and Vergne(see [3]). They discuss on the zero
points of a Killing vector field, the localization formula expresses the integral of an equivariantly closed differential form as an integral over the set of zeros of the Killing vector field. The de Rham model for equivariant cohomology give a deeper understanding of equivariant differential forms(see [1]). In [6], we introduce the equlvariant cohomology about two Killing vector fields and to establish a localization formulas on the set of zero points
$${\rm{Zero}}(X_{M}+\sqrt{-1}Y_{M})=\{x\in M \mid \langle X_{M}(x), Y_{M}(x)\rangle=0, |Y_{M}(x)|=|X_{M}(x)|\}.$$
For gaining a deeper understanding of equlvariant cohomology about two Killing vector fields, we introduce the Cartan model for equlvariant cohomology about two Killing vector fields(see [7]).

In this article, we will to establish a localization formulas of equlvariant cohomology about two Killing vector fields on the set of zero points
$${\rm{Zero}}(X_{M}-\sqrt{-1}Y_{M})=\{x\in M \mid |Y_{M}(x)|=|X_{M}(x)|=0 \}.$$
We will see that the set of zero points ${\rm{Zero}}(X_{M}-\sqrt{-1}Y_{M})$ is smaller and more basic.
As application, we use the localization formulas to get formulas about characteristic numbers and to get a Duistermaat-Heckman type formula on symplectic manifold.

\section{Equlvariant cohomology by two Killing vector fields}
First, let us review the definition of equlvariant cohomology about two Killing vector fields.
Let $M$ be a smooth closed oriented manifold. Let $G$ be
a compact Lie group acting smoothly on $M$, and let $\mathfrak{g}$
be its Lie algebra. Let $g^{TM}$ be a $G$-invariant metric on $TM$.
Let $\Omega^{*}(M)$ be the space of smooth differetial forms on $M$, the
de Rham complex is $(\Omega^{*}(M),d)$.
Let $\Omega^{*}(M)\otimes_{\mathbb{R}}\mathbb{C}$ be the space of smooth complex-valued differetial forms on $M$.
If $X,Y\in\mathfrak{g}$, let $X_{M} ,Y_{M}$ be the corresponding
smooth vector field on $M$ given by
$$(X_{M}f)(x)=\frac{d}{dt}f(\exp(-tX)\cdot x)\mid_{t=0}.$$
If $X,Y\in\mathfrak{g}$, then $X_{M}
,Y_{M}$ are Killing vector field.
Let $L_{X_{M}}$ be the Lie
derivative of $X_{M}$ on $\Omega^{*}(M)$, $i_{X_{M}}$ be the
interior multiplication induced by the contraction of $X_{M}$.\par

Set
$$L_{X_{M}+\sqrt{-1}Y_{M}}\doteq L_{X_{M}}+\sqrt{-1}L_{Y_{M}}$$
be the operator on
$\Omega^{*}(M)\otimes_{\mathbb{R}}\mathbb{C}$.\par

Set
$$i_{X_{M}+\sqrt{-1}Y_{M}}\doteq i_{X_{M}}+\sqrt{-1}i_{Y_{M}}$$
be the interior multiplication induced by the contraction of
$X_{M}+\sqrt{-1}Y_{M}$. It is also a operator on
$\Omega^{*}(M)\otimes_{\mathbb{R}}\mathbb{C}$.\par

Set$$d_{X+\sqrt{-1}Y}=d+i_{X_{M}+\sqrt{-1}Y_{M}}.$$

So
$$d_{X+\sqrt{-1}Y}^{2}=L_{X_{M}}+\sqrt{-1}L_{Y_{M}}=L_{X_{M}+\sqrt{-1}Y_{M}}.$$

Let$$\Omega_{X_{M}+\sqrt{-1}Y_{M}}^{*}(M)=\{\omega\in\Omega^{*}(M)\otimes_{\mathbb{R}}\mathbb{C}:L_{X_{M}+\sqrt{-1}Y_{M}}\omega=0\}$$
be the space of smooth $(X_{M}+\sqrt{-1}Y_{M})$-invariant forms on
$M$. Then we get a complex $(\Omega_{X_{M}+\sqrt{-1}Y_{M}}^{*}(M),
d_{X+\sqrt{-1}Y})$. We call a form $\omega$ is
$d_{X+\sqrt{-1}Y}$-closed if $d_{X+\sqrt{-1}Y}\omega=0$. The corresponding cohomology
group
$$H^{*}_{X+\sqrt{-1}Y}(M)=\frac{{\rm{Ker}}d_{X+\sqrt{-1}Y}|_{\Omega_{X+\sqrt{-1}Y}^{*}(M)}}{{\rm{Im}}d_{X+\sqrt{-1}Y}|_{\Omega_{X+\sqrt{-1}Y}^{*-1}(M)}}$$
is called the equivariant cohomology associated with $X_{M}+\sqrt{-1}Y_{M}$. By the same way, we can define the equivariant cohomology about two vector fields (not Killing vector fields). We can see that, if we set $Y_{M}=0$, then we get the equivariant cohomology as normal.

For any $\omega\in\Omega^{*}(M)\otimes_{\mathbb{R}}\mathbb{C}$, we can write it by $\xi+\sqrt{-1}\eta$, where $\xi,\eta\in\Omega^{*}(M)$. By the definition of $d_{X+\sqrt{-1}Y}$-closed forms, we have $\omega=\xi+\sqrt{-1}\eta$ is $d_{X+\sqrt{-1}Y}$-closed if and only if $d\xi+i_{X_{M}}\xi-i_{Y_{M}}\eta=0$ and $d\eta+i_{X_{M}}\eta+i_{Y_{M}}\xi=0$. For a special case, we have the following result

\begin{lemma}
$\omega=\xi+\sqrt{-1}\eta\in\Omega^{*}(M)\otimes_{\mathbb{R}}\mathbb{C}$ with $\xi,\eta$ are $m$-forms, then $\omega$ is $d_{X+\sqrt{-1}Y}$-closed if and only if $d\xi=0,d\eta=0$ and $i_{X_{M}}\xi=i_{Y_{M}}\eta, \ i_{X_{M}}\eta=-i_{Y_{M}}\xi$.
\end{lemma}
\begin{proof}
For $\omega=\xi+\sqrt{-1}\eta$, by the definition of $d_{X+\sqrt{-1}Y}$-closed forms, we have
$$d\xi+i_{X_{M}}\xi-i_{Y_{M}}\eta=0,\ \ d\eta+i_{X_{M}}\eta+i_{Y_{M}}\xi=0,$$
and because $\xi,\eta$ are $m$-forms, they have the same degree, so
$d\xi=0,d\eta=0$ and $i_{X_{M}}\xi=i_{Y_{M}}\eta, \ i_{X_{M}}\eta=-i_{Y_{M}}\xi$.

If $d\xi=0,d\eta=0$ and $i_{X_{M}}\xi=i_{Y_{M}}\eta, \ i_{X_{M}}\eta=-i_{Y_{M}}\xi$; then we have
$$d\xi+i_{X_{M}}\xi-i_{Y_{M}}\eta=0,\ \ d\eta+i_{X_{M}}\eta+i_{Y_{M}}\xi=0,$$
so $\omega=\xi+\sqrt{-1}\eta$ is $d_{X+\sqrt{-1}Y}$-closed forms.
\end{proof}

The condition $i_{X_{M}}\xi=i_{Y_{M}}\eta, \ i_{X_{M}}\eta=-i_{Y_{M}}\xi$ looks like the Cauchy-Riemann condition about holomorhpic functions.
\begin{example}
If $f=u+\sqrt{-1}v$ is a holomorhpic functions on $\mathbb{C}$, by the Cauchy-Riemann condition one have
$$\frac{\partial u}{\partial x}=\frac{\partial v}{\partial y}, \ \frac{\partial u}{\partial y}=-\frac{\partial v}{\partial x}.$$
Set $M=\mathbb{C}$, let $X_{M}=\frac{\partial}{\partial x}, Y_{M}=\frac{\partial}{\partial y}$, so by the Cauchy-Riemann condition we have
$$i_{\frac{\partial}{\partial x}}du=i_{\frac{\partial}{\partial y}}dv, \  i_{\frac{\partial}{\partial x}}dv=-i_{\frac{\partial}{\partial y}}du$$
Then by Lemma 1., $df$ is a $d_{X+\sqrt{-1}Y}$-closed forms.
\end{example}

\section{Some special $d_{X+\sqrt{-1}Y}$-closed forms}
In this section, we will give four special $d_{X+\sqrt{-1}Y}$-closed forms, $d_{X+\sqrt{-1}Y}(X^{'}+\sqrt{-1}Y^{'})$, $d_{X+\sqrt{-1}Y}(Y^{'}-\sqrt{-1}X^{'})$, $d_{X+\sqrt{-1}Y}(X^{'}-\sqrt{-1}Y^{'})$ and $d_{X+\sqrt{-1}Y}(Y^{'}+\sqrt{-1}X^{'})$.
\begin{lemma}
If $X,Y\in\mathfrak{g}$, let $X_{M} ,Y_{M}$ be the corresponding
smooth vector field on $M$, $X^{'}, Y^{'}$ be the 1-form on $M$
which is dual to $X_{M} ,Y_{M}$ by the metric $g^{TM}$, then
$$L_{X_{M}}Y^{'}+L_{Y_{M}}X^{'}=0$$
\end{lemma}
\begin{proof}
Because
$$(L_{X_{M}}\omega)(Z)=X_{M}(\omega(Z))-\omega([X_{M},Z])$$
here $Z\in\Gamma(TM)$, So we get
$$(L_{X_{M}}Y^{'})(Z)=X_{M}\langle Y_{M},Z\rangle-\langle[X_{M},Z],Y_{M}\rangle$$
$$(L_{Y_{M}}X^{'})(Z)=Y_{M}\langle X_{M},Z\rangle-\langle[Y_{M},Z],X_{M}\rangle.$$
Because $X_{M},Y_{M}$ are Killing vector fields, so(see [11])
\begin{align*}
X_{M}\langle Y_{M},Z\rangle
&=\langle L_{X_{M}}Y_{M},Z\rangle+\langle Y_{M},L_{X_{M}}Z\rangle\\
&=\langle[X_{M},Y_{M}],Z\rangle+\langle Y_{M},[X_{M},Z]\rangle
\end{align*}
\begin{align*}
Y_{M}\langle X_{M},Z\rangle
&=\langle L_{Y_{M}}X_{M},Z\rangle+\langle X_{M},L_{Y_{M}}Z\rangle\\
&=\langle[Y_{M},X_{M}],Z\rangle+\langle X_{M},[Y_{M},Z]\rangle
\end{align*}
then we get
$$(L_{X_{M}}Y^{'}+L_{Y_{M}}X^{'})(Z)=\langle[X_{M},Y_{M}],Z\rangle+\langle[Y_{M},X_{M}],Z\rangle=0$$
\end{proof}

\begin{lemma}
If $X,Y\in\mathfrak{g}$, let $X_{M} ,Y_{M}$ be the corresponding
smooth vector field on $M$, $X^{'}, Y^{'}$ be the 1-form on $M$
which is dual to $X_{M} ,Y_{M}$ by the metric $g^{TM}$, then
\begin{description}
\item[1)] $d_{X+\sqrt{-1}Y}(X^{'}+\sqrt{-1}Y^{'})$
\item[2)] $d_{X+\sqrt{-1}Y}(Y^{'}-\sqrt{-1}X^{'})$
\end{description}
are the $d_{X+\sqrt{-1}Y}$-closed forms.
\end{lemma}
\begin{proof}
\begin{align*}
d_{X+\sqrt{-1}Y}^{2}(X^{'}+\sqrt{-1}Y^{'})
&=(L_{X_{M}}+\sqrt{-1}L_{Y_{M}})(X^{'}+\sqrt{-1}Y^{'})\\
&=L_{X_{M}}X^{'}-L_{Y_{M}}Y^{'}+\sqrt{-1}(L_{X_{M}}Y^{'}+L_{Y_{M}}X^{'})\\
&=0
\end{align*}
So $d_{X+\sqrt{-1}Y}(X^{'}+\sqrt{-1}Y^{'})$ is the
$d_{X+\sqrt{-1}Y}$-closed form;

\begin{align*}
d_{X+\sqrt{-1}Y}^{2}(Y^{'}-\sqrt{-1}X^{'})
&=(L_{X_{M}}+\sqrt{-1}L_{Y_{M}})(Y^{'}-\sqrt{-1}X^{'})\\
&=L_{X_{M}}Y^{'}+L_{Y_{M}}X^{'}+\sqrt{-1}(L_{Y_{M}}Y^{'}-L_{X_{M}}X^{'})\\
&=0
\end{align*}
So $d_{X+\sqrt{-1}Y}(Y^{'}-\sqrt{-1}X^{'})$ is the
$d_{X+\sqrt{-1}Y}$-closed form.
\end{proof}

\begin{lemma}
If $X,Y\in\mathfrak{g}$ and with $[X,Y]=0$, then $[X_{M},Y_{M}]=0$.
\end{lemma}
\begin{proof}
Because $[X,Y]=0$, so we have
$$\exp(-tX)\exp(-sY)=\exp(-sY)\exp(-tX)$$
where $s,t\in\mathbb{R}$, and for any $f\in C^{\infty}(M)$
$$([X_{M},Y_{M}]f)(p)=(X_{M}(Y_{M}f)-Y_{M}(X_{M}f))(p)$$
$$=\frac{\partial^{2}}{\partial t\partial s}\bigg|_{s=t=0}f\bigg(\exp(-tX)\exp(-sY)\cdot p\bigg)-\frac{\partial^{2}}{\partial s\partial t}\bigg|_{s=t=0}f\bigg(\exp(-sY)\exp(-tX)\cdot p\bigg)=0$$
So we get $[X_{M},Y_{M}]=0$.
\end{proof}

\begin{lemma}
If $X,Y\in\mathfrak{g}$ with $[X,Y]=0$, let $X_{M} ,Y_{M}$ be the corresponding
smooth vector field on $M$, $X^{'}, Y^{'}$ be the 1-form on $M$
which is dual to $X_{M} ,Y_{M}$ by the metric $g^{TM}$, then
$$L_{X_{M}}Y^{'}=0, \ L_{Y_{M}}X^{'}=0.$$
\end{lemma}
\begin{proof}
Because $[X,Y]=0$, by Lemma 4., $[X_{M} ,Y_{M}]=0$, then
$$(L_{X_{M}}Y^{'})(Z)=\langle[X_{M},Y_{M}],Z\rangle+\langle Y_{M},[X_{M},Z]\rangle-\langle[X_{M},Z],Y_{M}\rangle=0,$$
$$(L_{Y_{M}}X^{'})(Z)=\langle[Y_{M},X_{M}],Z\rangle+\langle X_{M},[Y_{M},Z]\rangle-\langle[Y_{M},Z],X_{M}\rangle=0.$$

\end{proof}

\begin{lemma}
If $X,Y\in\mathfrak{g}$ with $[X,Y]=0$, let $X_{M} ,Y_{M}$ be the corresponding
smooth vector field on $M$, $X^{'}, Y^{'}$ be the 1-form on $M$
which is dual to $X_{M} ,Y_{M}$ by the metric $g^{TM}$, then
\begin{description}
\item[1)] $d_{X+\sqrt{-1}Y}(X^{'}-\sqrt{-1}Y^{'})$
\item[2)] $d_{X+\sqrt{-1}Y}(Y^{'}+\sqrt{-1}X^{'})$
\end{description}
are the $d_{X+\sqrt{-1}Y}$-closed forms.
\end{lemma}
\begin{proof}
Because $[X,Y]=0$, by Lemma 5., we have $L_{X_{M}}Y^{'}=0, \ L_{Y_{M}}X^{'}=0$;
\begin{align*}
d_{X+\sqrt{-1}Y}^{2}(X^{'}-\sqrt{-1}Y^{'})
&=(L_{X_{M}}+\sqrt{-1}L_{Y_{M}})(X^{'}-\sqrt{-1}Y^{'})\\
&=L_{X_{M}}X^{'}+L_{Y_{M}}Y^{'}+\sqrt{-1}(L_{Y_{M}}X^{'}-L_{X_{M}}Y^{'})\\
&=0
\end{align*}
So $d_{X+\sqrt{-1}Y}(X^{'}-\sqrt{-1}Y^{'})$ is the
$d_{X+\sqrt{-1}Y}$-closed form.

\begin{align*}
d_{X+\sqrt{-1}Y}^{2}(Y^{'}+\sqrt{-1}X^{'})
&=(L_{X_{M}}+\sqrt{-1}L_{Y_{M}})(Y^{'}+\sqrt{-1}X^{'})\\
&=L_{X_{M}}Y^{'}-L_{Y_{M}}X^{'}+\sqrt{-1}(L_{X_{M}}X^{'}+L_{Y_{M}}Y^{'})\\
&=0
\end{align*}
So $d_{X+\sqrt{-1}Y}(Y^{'}+\sqrt{-1}X^{'})$ is the
$d_{X+\sqrt{-1}Y}$-closed form.

\end{proof}

\section{The set of zero points}
In [6], we have get that for any $\eta\in H^{*}_{X+\sqrt{-1}Y}(M)$ and $s\geq 0$, we have
$$\int_{M}\eta=\int_{M}\exp\{-s(d_{X+\sqrt{-1}Y}(X^{'}+\sqrt{-1}Y^{'}))\}\eta.$$
Here we will give the same results about $d_{X+\sqrt{-1}Y}(Y^{'}-\sqrt{-1}X^{'})$, $d_{X+\sqrt{-1}Y}(X^{'}-\sqrt{-1}Y^{'})$ and $d_{X+\sqrt{-1}Y}(Y^{'}+\sqrt{-1}X^{'})$.

\begin{lemma}
For any $\eta\in H^{*}_{X+\sqrt{-1}Y}(M)$ and $s\geq 0$, we have
\begin{description}
\item[1)] $\int_{M}\eta=\int_{M}\exp\{-s(d_{X+\sqrt{-1}Y}(Y^{'}-\sqrt{-1}X^{'}))\}\eta$,
\item[2)] When $[X,Y]=0$, then $\int_{M}\eta=\int_{M}\exp\{-s(d_{X+\sqrt{-1}Y}(X^{'}-\sqrt{-1}Y^{'}))\}\eta$,
\item[3)] When $[X,Y]=0$, then $\int_{M}\eta=\int_{M}\exp\{-s(d_{X+\sqrt{-1}Y}(Y^{'}+\sqrt{-1}X^{'}))\}\eta$.
\end{description}

\end{lemma}
\begin{proof}
For 1), because $$\frac{\partial}{\partial
s}\int_{M}\exp\{-s(d_{X+\sqrt{-1}Y}(Y^{'}-\sqrt{-1}X^{'}))\}\eta$$
$$=-\int_{M}(d_{X+\sqrt{-1}Y}(Y^{'}-\sqrt{-1}X^{'}))\exp\{-s(d_{X+\sqrt{-1}Y}(Y^{'}-\sqrt{-1}X^{'}))\}\eta$$
and by assumption we have
$$d_{X+\sqrt{-1}Y}\eta=0$$
$$d_{X+\sqrt{-1}Y}\exp\{-s(d_{X+\sqrt{-1}Y}(Y^{'}-\sqrt{-1}X^{'}))\}=0$$
So we get
$$(d_{X+\sqrt{-1}Y}(Y^{'}-\sqrt{-1}X^{'}))\exp\{-s(d_{X+\sqrt{-1}Y}(Y^{'}-\sqrt{-1}X^{'}))\}\eta$$
$$=d_{X+\sqrt{-1}Y}[(Y^{'}-\sqrt{-1}X^{'})\exp\{-s(d_{X+\sqrt{-1}Y}(Y^{'}-\sqrt{-1}X^{'}))\}\eta]$$
and by Stokes formula we have
$$\frac{\partial}{\partial
s}\int_{M}\exp\{-s(d_{X+\sqrt{-1}Y}(Y^{'}-\sqrt{-1}X^{'}))\}\eta=0$$
Then we get
$$\int_{M}\eta=\int_{M}\exp\{-s(d_{X+\sqrt{-1}Y}(Y^{'}-\sqrt{-1}X^{'}))\}\eta$$

For 2) and 3), when $[X,Y]=0$, we have
$$d_{X+\sqrt{-1}Y}\exp\{-s(d_{X+\sqrt{-1}Y}(X^{'}-\sqrt{-1}Y^{'}))\}=0,$$
$$d_{X+\sqrt{-1}Y}\exp\{-s(d_{X+\sqrt{-1}Y}(Y^{'}+\sqrt{-1}X^{'}))\}=0,$$
so by the same way as in 1), we can get the results.
\end{proof}

For
$$d_{X+\sqrt{-1}Y}(Y^{'}-\sqrt{-1}X^{'})
=d(Y^{'}-\sqrt{-1}X^{'})+\langle X_{M}+\sqrt{-1}Y_{M},
Y_{M}-\sqrt{-1}X_{M}\rangle$$ and $$\langle X_{M}+\sqrt{-1}Y_{M},
Y_{M}-\sqrt{-1}X_{M}\rangle=2\langle X_{M}, Y_{M}\rangle+\sqrt{-1}(|Y_{M}|^{2}-|X_{M}|^{2})$$

Set
$${\rm{Zero}}(Y_{M}-\sqrt{-1}X_{M})=\{x\in M \mid \langle X_{M}(x)+\sqrt{-1}Y_{M}(x),
Y_{M}(x)-\sqrt{-1}X_{M}(x)\rangle=0\}.$$
We can see that $${\rm{Zero}}(Y_{M}-\sqrt{-1}X_{M})=\{x\in M \mid \langle X_{M}(x), Y_{M}(x)\rangle=0, |Y_{M}(x)|=|X_{M}(x)|\},$$
This set of zero points ${\rm{Zero}}(Y_{M}-\sqrt{-1}X_{M})$ is the same as in [6]. This set of zero points is first discussed by H.Jacobowitz (see [8] and [9]).

For
$$d_{X+\sqrt{-1}Y}(X^{'}-\sqrt{-1}Y^{'})
=d(X^{'}-\sqrt{-1}Y^{'})+\langle X_{M}+\sqrt{-1}Y_{M},
X_{M}-\sqrt{-1}Y_{M}\rangle$$ and $$\langle X_{M}+\sqrt{-1}Y_{M},
X_{M}-\sqrt{-1}Y_{M}\rangle=|X_{M}|^{2}+|Y_{M}|^{2},$$
set
$${\rm{Zero}}(X_{M}-\sqrt{-1}Y_{M})=\{x\in M \mid \langle X_{M}(x)+\sqrt{-1}Y_{M}(x),
X_{M}(x)-\sqrt{-1}Y_{M}(x)\rangle=0\}.$$

For
$$d_{X+\sqrt{-1}Y}(Y^{'}+\sqrt{-1}X^{'})
=d(Y^{'}+\sqrt{-1}X^{'})+\langle X_{M}+\sqrt{-1}Y_{M},
Y_{M}+\sqrt{-1}X_{M}\rangle$$ and $$\langle X_{M}+\sqrt{-1}Y_{M},
Y_{M}+\sqrt{-1}X_{M}\rangle=\sqrt{-1}(|X_{M}|^{2}+|Y_{M}|^{2}),$$
set
$${\rm{Zero}}(Y_{M}+\sqrt{-1}X_{M})=\{x\in M \mid \langle X_{M}(x)+\sqrt{-1}Y_{M}(x),
Y_{M}(x)+\sqrt{-1}X_{M}(x)\rangle=0\}.$$

We can see that $${\rm{Zero}}(X_{M}-\sqrt{-1}Y_{M})={\rm{Zero}}(Y_{M}+\sqrt{-1}X_{M})=\{x\in M \mid |X_{M}(x)|=|Y_{M}(x)|=0\}.$$

So we get two kinds of zero points, the one is
$$\{x\in M \mid \langle X_{M}(x), Y_{M}(x)\rangle=0, |Y_{M}(x)|=|X_{M}(x)|\},$$
the other one is
$$\{x\in M \mid |X_{M}(x)|=|Y_{M}(x)|=0\};$$
obviously
$$\{x\in M \mid |X_{M}(x)|=|Y_{M}(x)|=0\}\subset\{x\in M \mid \langle X_{M}(x), Y_{M}(x)\rangle=0, |Y_{M}(x)|=|X_{M}(x)|\}.$$

\begin{corollery}
For any $\eta\in H^{*}_{X+\sqrt{-1}Y}(M)$ with $[X,Y]=0$ and $s\geq 0$, we have
\begin{description}
\item[1)] $$\int_{M}\exp\{-s(d_{X+\sqrt{-1}Y}(X^{'}+\sqrt{-1}Y^{'}))\}\eta$$
$$=\int_{M}\exp\{-s(d_{X+\sqrt{-1}Y}(X^{'}-\sqrt{-1}Y^{'}))\}\exp\{-s(d_{X+\sqrt{-1}Y}(X^{'}+\sqrt{-1}Y^{'}))\}\eta,$$
\item[2)] $$\int_{M}\exp\{-s(d_{X+\sqrt{-1}Y}(Y^{'}-\sqrt{-1}X^{'}))\}\eta$$
$$=\int_{M}\exp\{-s(d_{X+\sqrt{-1}Y}(X^{'}-\sqrt{-1}Y^{'}))\}\exp\{-s(d_{X+\sqrt{-1}Y}(Y^{'}-\sqrt{-1}X^{'}))\}\eta,$$
\item[3)] $$\int_{M}\exp\{-s(d_{X+\sqrt{-1}Y}(Y^{'}+\sqrt{-1}X^{'}))\}\eta$$
$$=\int_{M}\exp\{-s(d_{X+\sqrt{-1}Y}(X^{'}-\sqrt{-1}Y^{'}))\}\exp\{-s(d_{X+\sqrt{-1}Y}(Y^{'}+\sqrt{-1}X^{'}))\}\eta.$$
\end{description}
\end{corollery}
\begin{proof}
Because $\exp\{-s(d_{X+\sqrt{-1}Y}(X^{'}+\sqrt{-1}Y^{'}))\}\eta\in H^{*}_{X+\sqrt{-1}Y}(M)$, $\exp\{-s(d_{X+\sqrt{-1}Y}(Y^{'}-\sqrt{-1}X^{'}))\}\eta\in H^{*}_{X+\sqrt{-1}Y}(M)$,
$\exp\{-s(d_{X+\sqrt{-1}Y}(Y^{'}+\sqrt{-1}X^{'}))\}\eta\in H^{*}_{X+\sqrt{-1}Y}(M)$.
So by Lemma 7., we get the result.

\end{proof}

\begin{lemma}
For any $\eta\in H^{*}_{X+\sqrt{-1}Y}(M)$ and $[X,Y]=0$, if ${\rm{Zero}}(X_{M}-\sqrt{-1}Y_{M})=\emptyset$, then $\int_{M}\eta=0$.
\end{lemma}
\begin{proof}
Because $$d_{X+\sqrt{-1}Y}(X^{'}-\sqrt{-1}Y^{'})
=d(X^{'}-\sqrt{-1}Y^{'})+|X_{M}|^{2}+|Y_{M}|^{2}$$
so
$$\int_{M}\exp\{-s(d_{X+\sqrt{-1}Y}(X^{'}+\sqrt{-1}Y^{'}))\}\eta=\int_{M}\exp\{-s(|X_{M}|^{2}+|Y_{M}|^{2})\}\exp\{-sd(X^{'}-\sqrt{-1}Y^{'})\}\eta.$$
By ${\rm{Zero}}(X_{M}-\sqrt{-1}Y_{M})=\emptyset$, we can see easily that when $s\rightarrow+\infty$, the right hand side of the above equality is of exponential decay and so the result follows.
\end{proof}

\section{Localization formula on ${\rm{Zero}}(X_{M}-\sqrt{-1}Y_{M})$}
In the following section we denote ${\rm{Zero}}(X_{M}-\sqrt{-1}Y_{M})$ by $M_{0}$.
For simplicity, we assume that $M_{0}$ is the
connected submanifold of $M$, and $\mathcal{N}$ is the normal bundle
of $M_{0}$ about $M$.

Set $E$ is a G-equivariant vector bundle, if $\nabla^{E}$ is a
connection on $E$ which commutes with the action of $G$ on
$\Omega(M,E)$, we see that $$[\nabla^{E}, L^{E}_{X}]=0$$ for all
$X\in\mathfrak{g}$. Then we can get a moment map by
$$\mu^{E}(X)=L^{E}_{X}-[\nabla^{E},i_{X}]=L^{E}_{X}-\nabla^{E}_{X}$$
We known that if $y$ be the tautological section of the bundle
$\pi^{*}E$ over E, then the vertical component of $X_{E}$ may be
identified with $-\mu^{E}(X)y$(see [2] proposition 7.6). For the normal bundle $\mathcal{N}$ of $M_0$, the vector fields $X_{\mathcal{N}}$ and $Y_{\mathcal{N}}$ are vertical
and are given at the point $(x,y)\in M_{0}\times\mathcal{N}_{x}$ by
the vectors $-\mu^{\mathcal{N}}(X)y,
-\mu^{\mathcal{N}}(Y)y\in\mathcal{N}_{x}$.\par

If $E$ is the tangent bundle $TM$ and $\nabla^{TM}$ is Levi-Civita
connection, then we have
$$\mu^{TM}(X)Y=L_{X}Y-\nabla^{TM}_{X}Y=-\nabla^{TM}_{Y}X$$
We known that for any Killing vector field $X$, $\mu^{TM}(X)$ as
linear endomorphisms of $TM$ is skew-symmetric, $-\mu^{TM}(X)$
annihilates the tangent bundle $TM_{0}$ and induces a skew-symmetric
automorphism of the normal bundle $\mathcal{N}$(see [10] chapter II,
proposition 2.2 and theorem 5.3). The restriction of $\mu^{TM}(X)$
to $\mathcal{N}$ coincides with the moment endomorphism
$\mu^{\mathcal{N}}(X)$.

Now we construct a one-form $\alpha$ on $\mathcal{N}$:
$$Z\in \Gamma(T\mathcal{N})\rightarrow\alpha(Z)=\langle-\mu^{\mathcal{N}}(X)y,\nabla^{\mathcal{N}}_{Z}y\rangle-\sqrt{-1}\langle-\mu^{\mathcal{N}}(Y)y,\nabla^{\mathcal{N}}_{Z}y\rangle$$

Let $Z_{1},Z_{2}\in\Gamma(T\mathcal{N})$, we known
$d\alpha(Z_{1},Z_{2})=Z_{1}\alpha(Z_{2})-Z_{2}\alpha(Z_{1})-\alpha([Z_{1},Z_{2}])$,
so
\begin{align*}
d\alpha(Z_{1},Z_{2})
&=\langle-\nabla^{\mathcal{N}}_{Z_{1}}\mu^{\mathcal{N}}(X)y,\nabla^{\mathcal{N}}_{Z_{2}}y\rangle-\langle-\nabla^{\mathcal{N}}_{Z_{2}}\mu^{\mathcal{N}}(X)y,\nabla^{\mathcal{N}}_{Z_{1}}y\rangle\\
&-\sqrt{-1}\langle-\nabla^{\mathcal{N}}_{Z_{1}}\mu^{\mathcal{N}}(Y)y,\nabla^{\mathcal{N}}_{Z_{2}}y\rangle+\sqrt{-1}\langle-\nabla^{\mathcal{N}}_{Z_{2}}\mu^{\mathcal{N}}(Y)y,\nabla^{\mathcal{N}}_{Z_{1}}y\rangle\\
&+\langle-\mu^{\mathcal{N}}(X)y,R^{\mathcal{N}}(Z_{1},Z_{2})y\rangle-\sqrt{-1}\langle-\mu^{\mathcal{N}}(Y)y,R^{\mathcal{N}}(Z_{1},Z_{2})y\rangle\\
\end{align*}
Recall that $\nabla^{\mathcal{N}}$ is invariant under $L_{X}$ for
all $X\in\mathfrak{g}$, so that
$[\nabla^{\mathcal{N}},\mu^{\mathcal{N}}(X)]=0$,
$[\nabla^{\mathcal{N}},\mu^{\mathcal{N}}(Y)]=0$. And by $X,Y$ are
Killing vector field, we have $d\alpha$ equals
$$2\langle-(\mu^{\mathcal{N}}(X)-\sqrt{-1}\mu^{\mathcal{N}}(Y))
\cdot,\cdot\rangle+\langle-\mu^{\mathcal{N}}(X)y+\sqrt{-1}\mu^{\mathcal{N}}(Y)y,R^{\mathcal{N}}y\rangle$$
And by
$|X_{\mathcal{N}}|^{2}=\langle\mu^{\mathcal{N}}(X)y,\mu^{\mathcal{N}}(X)y\rangle$,
$|Y_{\mathcal{N}}|^{2}=\langle\mu^{\mathcal{N}}(Y)y,\mu^{\mathcal{N}}(Y)y\rangle$.
So We can get
\begin{align*}
d_{X_{\mathcal{N}}+\sqrt{-1}Y_{\mathcal{N}}}(X^{'}_{\mathcal{N}}-\sqrt{-1}Y^{'}_{\mathcal{N}})
&=d(X^{'}_{\mathcal{N}}-\sqrt{-1}Y^{'}_{\mathcal{N}})+\langle X_{\mathcal{N}}-\sqrt{-1}Y_{\mathcal{N}},X_{\mathcal{N}}+\sqrt{-1}Y_{\mathcal{N}}\rangle\\
&=-2\langle(\mu^{\mathcal{N}}(X)-\sqrt{-1}\mu^{\mathcal{N}}(Y))\cdot,\cdot\rangle\\
&+\langle-\mu^{\mathcal{N}}(X)y+\sqrt{-1}\mu^{\mathcal{N}}(Y)y,-\mu^{\mathcal{N}}(X)y-\sqrt{-1}\mu^{\mathcal{N}}(Y)y+R^{\mathcal{N}}y\rangle
\end{align*}

\begin{theorem}
Let $M$ be a smooth closed oriented manifold, $G$ be a compact Lie group
acting smoothly on $M$. For any $\eta\in H^{*}_{X+\sqrt{-1}Y}(M)$, $[X,Y]=0$, the following identity
hold:
$$\int_{M}\eta=\int_{M_{0}} \frac{\eta}{\rm{Pf}[\frac{-\mu^{\mathcal{N}}(X)-\sqrt{-1}\mu^{\mathcal{N}}(Y)+R^{\mathcal{N}}}{2\pi}]}$$
\end{theorem}
\begin{proof}
Here we use the method come from [5]. Set $s=\frac{1}{2t}$, so by Lemma 7. we get
$$\int_{M}\eta=\int_{M}\exp\{-\frac{1}{2t}(d_{X+\sqrt{-1}Y}(X^{'}-\sqrt{-1}Y^{'}))\}\eta$$
Let $V$ is a neighborhood of $M_{0}$ in $\mathcal{N}$. We identify a
tubular neighborhood of $M_{0}$ in $M$ with $V$. Set $V^{'}\subset
V$. When $t\rightarrow 0$, because
$$\langle X_{M}(x)+\sqrt{-1}Y_{M}(x), X_{M}(x)-\sqrt{-1}Y_{M}(x)\rangle=|X_{M}|^{2}+|Y_{M}|^{2}\neq0$$
out of $M_{0}$, so we have
$$\int_{M}\exp\{-\frac{1}{2t}(d_{X+\sqrt{-1}Y}(X^{'}-\sqrt{-1}Y^{'}))\}\eta\sim\int_{V^{'}}\exp\{-\frac{1}{2t}(d_{X+\sqrt{-1}Y}(X^{'}-\sqrt{-1}Y^{'}))\}\eta.$$
Because
$$\int_{V^{'}}\exp\{-\frac{1}{2t}(d_{X+\sqrt{-1}Y}(X^{'}-\sqrt{-1}Y^{'}))\}\eta=\int_{V^{'}}\exp\{-\frac{1}{2t}(d_{X_{\mathcal{N}}+\sqrt{-1}Y_{\mathcal{N}}}(X_{\mathcal{N}}^{'}-\sqrt{-1}Y_{\mathcal{N}}^{'}))\}\eta$$
then
$$\int_{V^{'}}\exp\{-\frac{1}{2t}(d_{X+\sqrt{-1}Y}(X^{'}-\sqrt{-1}Y^{'}))\}\eta=$$
$$\int_{V^{'}}\exp\{\frac{1}{t}\langle(\mu^{\mathcal{N}}(X)-\sqrt{-1}\mu^{\mathcal{N}}(Y))\cdot,\cdot\rangle+\frac{1}{2t}\langle\mu^{\mathcal{N}}(X)y-\sqrt{-1}\mu^{\mathcal{N}}(Y)y,R^{\mathcal{N}}y\rangle\}\eta$$
$$+\int_{V^{'}}\exp\{-\frac{1}{2t}\langle-\mu^{\mathcal{N}}(X)y+\sqrt{-1}\mu^{\mathcal{N}}(Y)y, -\mu^{\mathcal{N}}(X)y-\sqrt{-1}\mu^{\mathcal{N}}(Y)y\rangle\}\eta$$
By making the change of variables $y=\sqrt{t}y$, we find that the
above formula is equal to
$$t^{n}\int_{V^{'}}\exp\{\frac{1}{t}\langle(\mu^{\mathcal{N}}(X)-\sqrt{-1}\mu^{\mathcal{N}}(Y))\cdot,\cdot\rangle+\frac{1}{2}\langle\mu^{\mathcal{N}}(X)y-\sqrt{-1}\mu^{\mathcal{N}}(Y)y,R^{\mathcal{N}}y\rangle\}\eta$$
$$+\int_{V^{'}}\exp\{-\frac{1}{2}\langle-\mu^{\mathcal{N}}(X)y+\sqrt{-1}\mu^{\mathcal{N}}(Y)y, -\mu^{\mathcal{N}}(X)y-\sqrt{-1}\mu^{\mathcal{N}}(Y)y\rangle\}\eta$$
we known that
$$\frac{(\frac{\langle(\mu^{\mathcal{N}}(X)-\sqrt{-1}\mu^{\mathcal{N}}(Y))\cdot,\cdot\rangle}{t})^{n}}{n!}=(\rm{Pf}(\mu^{\mathcal{N}}(X)-\sqrt{-1}\mu^{\mathcal{N}}(Y)))dy$$
here dy is the volume form of the submanifold $M_{0}$. Because $$(\rm{Pf}(\mu^{\mathcal{N}}(X)-\sqrt{-1}\mu^{\mathcal{N}}(Y)))^{2}=det(\mu^{\mathcal{N}}(X)-\sqrt{-1}\mu^{\mathcal{N}}(Y)),$$
let $n$ be the dimension of $M_{0}$, then we get
$$=\int_{V^{'}}\exp\{\frac{1}{2}\langle\mu^{\mathcal{N}}(X)y-\sqrt{-1}\mu^{\mathcal{N}}(Y)y,R^{\mathcal{N}}y\rangle\}\eta[\det(\mu^{\mathcal{N}}(X)-\sqrt{-1}\mu^{\mathcal{N}}(Y))]^{\frac{1}{2}}dy_{1}\wedge...\wedge dy_{n}$$
$$+\int_{V^{'}}\exp\{-\frac{1}{2}\langle-\mu^{\mathcal{N}}(X)y+\sqrt{-1}\mu^{\mathcal{N}}(Y)y, -\mu^{\mathcal{N}}(X)y-\sqrt{-1}\mu^{\mathcal{N}}(Y)y\rangle\}\eta$$
Because by $[X,Y]=0$ we have $[\mu^{TM}(X),\mu^{TM}(Y)]=0$.
And by $-\mu^{\mathcal{N}}(X)-\sqrt{-1}\mu^{\mathcal{N}}(Y)$,
$R^{\mathcal{N}}$ are skew-symmetric, so we get
$$=\int_{V^{'}}\exp\{-\frac{1}{2}\langle-\mu^{\mathcal{N}}(X)y+\sqrt{-1}\mu^{\mathcal{N}}(Y)y,-\mu^{\mathcal{N}}(X)y-\sqrt{-1}\mu^{\mathcal{N}}(Y)y+R^{\mathcal{N}}y\rangle\}dy_{1}\wedge...\wedge dy_{n}$$
$$\cdot[\det(\mu^{\mathcal{N}}(X)-\sqrt{-1}\mu^{\mathcal{N}}(Y))]^{\frac{1}{2}}\eta$$
$$=\int_{M_{0}}(2\pi)^{n}[\det(\mu^{\mathcal{N}}(X)-\sqrt{-1}\mu^{\mathcal{N}}(Y))]^{-\frac{1}{2}}[\det(-\mu^{\mathcal{N}}(X)-\sqrt{-1}\mu^{\mathcal{N}}(Y)+R^{\mathcal{N}})]^{-\frac{1}{2}}$$
$$\cdot[\det(\mu^{\mathcal{N}}(X)-\sqrt{-1}\mu^{\mathcal{N}}(Y))]^{\frac{1}{2}}\eta$$
$$=\int_{M_{0}}(2\pi)^{n}[\det(-\mu^{\mathcal{N}}(X)-\sqrt{-1}\mu^{\mathcal{N}}(Y)+R^{\mathcal{N}})]^{-\frac{1}{2}}\eta$$
$$=\int_{M_{0}} \frac{\eta}{\rm{Pf}[\frac{-\mu^{\mathcal{N}}(X)-\sqrt{-1}\mu^{\mathcal{N}}(Y)+R^{\mathcal{N}}}{2\pi}]}$$
\end{proof}
By Theorem 1.,we can get the localization formulas of Berline and Vergne(see [2] or [3]).
\begin{corollery}[N.Berline and M.Vergne]
Let $M$ be a smooth closed oriented manifold, $G$ be a compact Lie
group acting smoothly on $M$. For any $\eta\in H^{*}_{X}(M)$, the following identity
hold:
$$\int_{M}\eta=\int_{M_{0}} \frac{\eta}{\rm{Pf}[\frac{-\mu^{\mathcal{N}}(X)+R^{\mathcal{N}}}{2\pi}]}$$
\end{corollery}
\begin{proof}
By Theorem 1., we set $Y=0$, then we get the result.
\end{proof}

\section{Application in Characteristic Numbers}
As in [6], we will give the application of the localization formula about two Killing vector field in characteristic numbers. So let's recall the Chern-Weil theory(see [12]) about equivariant connection and equivariant curvature without proof(see [6] for proof) .

Let $M$ be an even dimensional compact oriented manifold without boundary, $G$ be a compact Lie group
acting smoothly on $M$ and $\mathfrak{g}$
be its Lie algebra. Let $g^{TM}$ be a $G$-invariant Riemannian metric on $TM$, $\nabla^{TM}$ is the Levi-Civita connection associated to $g^{TM}$. Here $\nabla^{TM}$ is a $G$-invariant connection, we see that $[\nabla^{TM},L_{X_{M}}]=0$ for all $X\in\mathfrak{g}$.\par

The equivariant connection $\widetilde{\nabla}^{TM}$ is the operator on $\Omega^{*}(M,TM)$ corresponding to a $G$-invariant connection $\nabla^{TM}$ is defined by the formula
$$\widetilde{\nabla}^{TM}=\nabla^{TM}+i_{X_{M}+\sqrt{-1}Y_{M}}$$
here $X_{M} ,Y_{M}$ be the smooth vector field on $M$ corresponded to $X,Y\in\mathfrak{g}$.
\begin{lemma}
The operator $\widetilde{\nabla}^{TM}$ preserves the space $\Omega^{*}_{X_{M}+\sqrt{-1}Y_{M}}(M,TM)$ which is the space of smooth $(X_{M}+\sqrt{-1}Y_{M})$-invariant forms with values in $TM$.
\end{lemma}

We will also denote the restriction of $\widetilde{\nabla}^{TM}$ to $\Omega^{*}_{X_{M}+\sqrt{-1}Y_{M}}(M,TM)$ by $\widetilde{\nabla}^{TM}$.

The equivariant curvature $\widetilde{R}^{TM}$ of the equivariant connection $\widetilde{\nabla}^{TM}$ is defined by the formula(see [2])
$$\widetilde{R}^{TM}=(\widetilde{\nabla}^{TM})^{2}-L_{X_{M}}-\sqrt{-1}L_{Y_{M}}$$
It is the element of $\Omega^{*}_{X_{M}+\sqrt{-1}Y_{M}}(M,End(TM))$. We see that
\begin{align*}
\widetilde{R}^{TM}
&=(\nabla^{TM}+i_{X_{M}+\sqrt{-1}Y_{M}})^{2}-L_{X_{M}}-\sqrt{-1}L_{Y_{M}}\\
&=R^{TM}+[\nabla^{TM},i_{X_{M}+\sqrt{-1}Y_{M}}]-L_{X_{M}}-\sqrt{-1}L_{Y_{M}}\\
&=R^{TM}-\mu^{TM}(X)-\sqrt{-1}\mu^{TM}(Y)
\end{align*}

\begin{lemma}
The equivariant curvature $\widetilde{R}^{TM}$ satisfies the equvariant Bianchi formula $$\widetilde{\nabla}^{TM}\widetilde{R}^{TM}=0$$
\end{lemma}

Now we construct the equivariant characteristic forms by $\widetilde{R}^{TM}$. If $f(x)$ is a polynomial in the indeterminate $x$, then $f(\widetilde{R}^{TM})$ is an element of $\Omega^{*}_{X_{M}+\sqrt{-1}Y_{M}}(M,End(TM))$. We use the trace map
$${\rm Tr}: \Omega^{*}_{X_{M}+\sqrt{-1}Y_{M}}(M,End(TM))\rightarrow\Omega^{*}_{X_{M}+\sqrt{-1}Y_{M}}(M)$$
to obtain an element of $\Omega^{*}_{X_{M}+\sqrt{-1}Y_{M}}(M)$, which we call an equivariant characteristic form.
\begin{lemma}
The equivariant differential form ${\rm Tr}(f(\widetilde{R}^{TM}))$ is $d_{X_{M}+\sqrt{-1}Y_{M}}$-closed, and its equivariant cohomology class is independent of the choice of the G-invariant connection $\nabla^{TM}$.
\end{lemma}

As an application of Theorem 1., we can get the following localization formulas for characteristic numbers
\begin{theorem}
Let $M$ be an $2m$-dim compact oriented manifold without boundary, $G$ be a compact Lie group
acting smoothly on $M$ and $\mathfrak{g}$ be its Lie algebra. Let $X,Y\in\mathfrak{g}$, and $X_{M} ,Y_{M}$ be the corresponding
smooth vector field on $M$, $M_{0}={\rm{Zero}}(X_{M}-\sqrt{-1}Y_{M})$. If $f(x)$ is a polynomial, then we have
$$\int_{M}{\rm Tr}(f(\widetilde{R}^{TM}))=\int_{M_{0}}\frac{{\rm Tr}(f(\widetilde{R}^{TM}))}{\rm{Pf}[\frac{-\mu^{\mathcal{N}}(X)-\sqrt{-1}\mu^{\mathcal{N}}(Y)+R^{\mathcal{N}}}{2\pi}]}$$
\end{theorem}
\begin{proof}
By Lemma 10., we have ${\rm Tr}(f(\widetilde{R}^{TM}))$ is $d_{X_{M}+\sqrt{-1}Y_{M}}$-closed. And by Theorem 1., we get the result.
\end{proof}

Now we use the detaminate map
$${\rm det}: \Omega^{*}_{X_{M}+\sqrt{-1}Y_{M}}(M,End(TM))\rightarrow\Omega^{*}_{X_{M}+\sqrt{-1}Y_{M}}(M)$$
to obtain an element of $\Omega^{*}_{X_{M}+\sqrt{-1}Y_{M}}(M)$.

\begin{lemma}
The equivariant differential form ${\rm Pf}(-\widetilde{R}^{TM})$ is $d_{X_{M}+\sqrt{-1}Y_{M}}$-closed, and its equivariant cohomology class is independent of the choice of the G-invariant connection $\nabla^{TM}$.
\end{lemma}
\begin{proof}
Because $\det A=\exp({\rm{Tr}}(\log(A)))$, so $$\det (-\widetilde{R}^{TM})=\exp({\rm{Tr}}(\log(-\widetilde{R}^{TM}))).$$
and we know that $\det (-\widetilde{R}^{TM})=({\rm Pf}(-\widetilde{R}^{TM}))^{2}$, by Lemma 11., we get the result.
\end{proof}

\begin{theorem}
Let $M$ be an $2m$-dim compact oriented manifold without boundary, $G$ be a compact Lie group
acting smoothly on $M$ and $\mathfrak{g}$ be its Lie algebra. Let $X,Y\in\mathfrak{g}$, and $X_{M} ,Y_{M}$ be the corresponding
smooth vector field on $M$, $M_{0}={\rm{Zero}}(X_{M}-\sqrt{-1}Y_{M})$. Then we have
$$\int_{M}{\rm Pf}(-\widetilde{R}^{TM})=\int_{M_{0}}\frac{{\rm Pf}(-\widetilde{R}^{TM})}{\rm{Pf}[\frac{-\mu^{\mathcal{N}}(X)-\sqrt{-1}\mu^{\mathcal{N}}(Y)+R^{\mathcal{N}}}{2\pi}]}$$
\end{theorem}
\begin{proof}
Because ${\rm Pf}(-\widetilde{R}^{TM})$ is $d_{X_{M}+\sqrt{-1}Y_{M}}$-closed and by Theorem 1., we get the result.
\end{proof}

\section{Application in Symplectic Manifolds}
Let $(M,\omega)$ be a smooth closed symplectic manifold, $\omega$ is a closed nondegenerate 2-form with $d\omega=0$(see [4]). Let $G$ be a connected compact Lie group acting on $M$ via symplectomorphism, i.e. $$L_{X}\omega=0$$ for $\forall X\in\mathfrak{g}$, here $\mathfrak{g}$ be its Lie algebra. If $X,Y\in\mathfrak{g}$, let $X_{M} ,Y_{M}$ be the corresponding
smooth vector field on $M$ given by
$$(X_{M}f)(x)=\frac{d}{dt}f(\exp(-tX)\cdot x)\mid_{t=0}.$$

By the symplectic form $\omega$ there is a isomorphism between vector fields and 1-form on $M$, i.e.
$$\Gamma(TM)\rightarrow\Omega^{1}(M): X_{M}\mapsto i_{X_{M}}\omega$$

For $H\in C^{\infty}(M)$, then a vector field $X^{H}$ on $M$ is called a Hamiltonian vector field with the energy function $H$, if for $X^{H}$ we have $i_{X^{H}}\omega=dH$.

We can also define the equivariant cohomology associated with $X+\sqrt{-1}Y$ on symplectic manifold in the same way as in Section 1.

Here we define the equivariant extension of the symplectic form by
$$\omega-H_{X}-\sqrt{-1}H_{Y}$$
where $dH_{X}=i_{X_{M}}\omega$, $dH_{Y}=i_{Y_{M}}\omega$.

\begin{lemma}
The equivariant symplectic form $\omega-H_{X}-\sqrt{-1}H_{Y}$ is a $d_{X+\sqrt{-1}Y}$-closed form.
\end{lemma}
\begin{proof}
\begin{align*}
d_{X+\sqrt{-1}Y}(\omega-H_{X}-\sqrt{-1}H_{Y})
&=(d+i_{X_{M}}+\sqrt{-1}i_{Y_{M}})(\omega-H_{X}-\sqrt{-1}H_{Y})\\
&=d\omega-dH_{X}-\sqrt{-1}dH_{Y}+i_{X_{M}}\omega+\sqrt{-1}i_{Y_{M}}\omega\\
&=d\omega\\
&=0
\end{align*}
\end{proof}

Since $d(H_{X}+\sqrt{-1}H_{Y})=i_{X_{M}}\omega+\sqrt{-1}i_{Y_{M}}\omega$, the set of points where the one-form $d(H_{X}+\sqrt{-1}H_{Y})$ vanishes coincides with the zero set ${\rm{Zero}}(X_{M}-\sqrt{-1}Y_{M})$.

\begin{theorem}
Let $(M,\omega)$ be a compact symplectic manifold, and let $G$  be a connected compact Lie group acting on M and $\mathfrak{g}$ be its Lie algebra. Also assume $M$ be a Riemannian manifold with $G$-invariant Riemannian metric $g^{TM}$. Let $X,Y\in\mathfrak{g}$, and $X_{M} ,Y_{M}$ be the corresponding smooth vector field on $M$, $M_{0}={\rm{Zero}}(X_{M}-\sqrt{-1}Y_{M})$. Then we have
$$\int_{M}\exp(-H_{X}-\sqrt{-1}H_{Y})\frac{\omega^{n}}{n!}=\int_{M_{0}}\frac{\exp(\omega)}{\rm{Pf}[\frac{-\mu^{\mathcal{N}}(X)-\sqrt{-1}\mu^{\mathcal{N}}(Y)+R^{\mathcal{N}}}{2\pi}]}$$
\end{theorem}
\begin{proof}
By Lemma 13., $\omega-H_{X}-\sqrt{-1}H_{Y}$ is a $d_{X+\sqrt{-1}Y}$-closed form; and
$$\int_{M}\exp(\omega-H_{X}-\sqrt{-1}H_{Y})=\int_{M}\exp(-H_{X}-\sqrt{-1}H_{Y})\exp(\omega)=\int_{M}\exp(-H_{X}-\sqrt{-1}H_{Y})\frac{\omega^{n}}{n!}.$$
Note that $\exp(-H_{X}-\sqrt{-1}H_{Y})=1$ on $M_{0}$.
Then by Theorem 1., we get the result.
\end{proof}
Obviously, this is a Duistermaat-Heckman type formula.

\end{CJK}
\end{document}